\documentclass[12pt, reqno]{amsart}%
\usepackage{amsmath, amsthm, amscd, amsfonts, amssymb, graphicx, color}
\usepackage[bookmarksnumbered, colorlinks, plainpages]{hyperref}%
\usepackage{amsmath}%
\usepackage{amsfonts}%
\usepackage{amssymb}%
\usepackage{graphicx}
\providecommand{\U}[1]{\protect \rule{.1in}{.1in}}
\newtheorem{theorem}{Theorem}[section]

\newtheorem{proposition}[theorem]{Proposition}
\newtheorem{corollary}[theorem]{Corollary}
\theoremstyle{definition}
\newtheorem{definition}[theorem]{Definition}

\theoremstyle{remark}

\numberwithin{equation}{section}
\begin{document}
\title[Mannheim curves in three dimensional Lie groups]{On Mannheim partner curves in three dimensional Lie groups}
\author{\.{I}sma\.{I}l G\"{o}k}
\address{Ankara University, Faculty of Science, Department of Mathematics, 06100,
Tando\~{g}an, Ankara, Turkey}
\email{igok@science.ankara.edu.tr}
\author{O. Zek\.{I} Okuyucu}
\address{Bilecik \c{S}eyh Edebali University, Faculty of Science and Arts, Department
of Mathematics, Bilecik, Turkey}
\email{osman.okuyucu@bilecik.edu.tr}
\author{Nejat Ekmekc\.{I}}
\address{Ankara University, Faculty of Science, Department of Mathematics, 06100,
Tando\~{g}an, Ankara, Turkey}
\email{nekmekci@science.ankara.edu.tr}
\author{Yusuf Yayl\i}
\address{Ankara University, Faculty of Science, Department of Mathematics, 06100,
Tando\~{g}an, Ankara, Turkey}
\email{yayli@science.ankara.edu.tr}
\thanks{This paper is in final form and no version of it will be submitted for
publication elsewhere.}
\date{August 14, 2012}

\begin{abstract}
In this paper, we define Mannheim partner curves in a three dimensional Lie
group $G$ with a bi-invariant metric. And then the main result in this paper
is given as (Theorem \ref{teo 3.2}): A curve $\alpha:I\subset
\mathbb{R\rightarrow}G$ with the Frenet apparatus $\left \{  T,N,B,\varkappa
,\tau \right \}  $ is a Mannheim partner curve if and only if%
\[
\lambda \varkappa \left(  1+H^{2}\right)  =1
\]
where $\lambda$ is constants and $H$ is the harmonic curvature
function of the curve $\alpha.$

\end{abstract}
\maketitle

\setcounter{page}{1}

\setcounter{page}{1}


\setcounter{page}{1}


\section{Introduction}

In the classical diferential geometry of curves, J. Bertrand studied curves in
Euclidean 3-space whose principal normals are the principal normals of another
curve. In (see \cite{bert}) he showed that a necessary and sufficient
condition for the existence of such a second curve is that a linear
relationship with constant coefficients shall exist between the first and
second curvatures of the given original curve. In other word, if we denote
first and second curvatures of a given curve by $k_{1}$ and $k_{2}$
respectively, then for $\lambda,\mu$ $\in \mathbb{R}$ we have $\lambda
k_{1}+\mu k_{2}=1$. Since the time of Bertrand's paper, pairs of curves of
this kind have been called \textit{Conjugate Bertrand} \textit{Curves}, or
more commonly \textit{Bertrand Curves} (see \cite{kuh})$.$

Another kind of associated curve whose principal normal vector field is the
binormal vector field of another curve.is called Mannheim curve. Mannheim
partner curves was studied by Liu and Wang (see \cite{liu}) in Euclidean $3-$
space and in the Minkowski $3-$space. After these papers lots of papers were
published about Mannheim curves in Euclidean $3-$space, Minkowski $3-$space,
dual $3-$space and Galilean spaces (see \cite{orbay, karacan, handan, sidika,
tosun}). Matsuda and Yorozu \cite{matsuda} gave a definition of generalized
Mannheim curve in Euclidean $4-$space. They show some characterizations and
examples of generalized Mannheim curves. Ersoy \emph{et.al. }gave a definition
of generalized Mannheim curve in Minkowski $4-$space.


The degenarete semi-Riemannian geometry of Lie group is studied by
\c{C}\"{o}ken and \c{C}ift\c{c}i \cite{coken}. Moreover, they obtanied a
naturally reductive homogeneous semi-Riemannian space using the Lie group.
Then \c{C}ift\c{c}i \cite{ciftci} defined general helices in three dimensional
Lie groups with a bi-invariant metric and obtained a generalization of
Lancret's theorem and gave a relation between the geodesics of the so-called
cylinders and general helices.


Recently, \textit{Izumiya and Takeuchi} \cite{izu} have introduced the concept
of slant helix in Euclidean $3$-space.\ A slant helix in Euclidean space
$\mathbb{E}^{3}$ was defined by the property that its principal normal vector
field makes a constant angle with a fixed direction. Also, Izumiya and
Takeuchi showed that $\alpha$ is a slant helix if and only if the geodesic
curvature of spherical image of principal normal indicatrix $\left(  N\right)
$ of a space curve $\alpha$
\[
\sigma_{N}\left(  s\right)  =\left(  \frac{\varkappa^{2}}{\left(
\varkappa^{2}+\tau^{2}\right)  ^{3/2}}\left(  \frac{\tau}{\varkappa}\right)
^{\prime}\right)  \left(  s\right)
\]
is a constant function .


Harmonic curvature functions were defined earlier by \"{O}zdamar and Hac\i
saliho\u{g}lu \cite{ozdamar}. Recently, many studies have been reported on
generalized helices and slant helices using the harmonic curvatures in
Euclidean spaces and Minkowski spaces \cite{gok, camci, kulahci}. Then,
Okuyucu et al. \cite{zeki} defined slant helices in three dimensional Lie
groups with a bi-invariant metric and obtained some characterizations using
their harmonic curvature function.


In this paper, first of all, we define Mannheim partner curves in a three
dimensional Lie group $G$ with a bi-invariant metric and we obtain the
necessary and sufficient conditions for the Mannheim partner curves in a three
dimensional Lie group $G.$


\section{Preliminaries}

Let $G$ be a Lie group with a bi-invariant metric $\left \langle \text{
},\right \rangle $ and $D$ be the Levi-Civita connection of Lie group $G.$ If
$\mathfrak{g}$ denotes the Lie algebra of $G$ then we know that $\mathfrak{g}
$ is issomorphic to $T_{e}G$ where $e$ is neutral element of $G.$ If
$\left \langle \text{ },\right \rangle $ is a bi-invariant metric on $G$ then we
have%
\begin{equation}
\left \langle X,\left[  Y,Z\right]  \right \rangle =\left \langle \left[
X,Y\right]  ,Z\right \rangle \label{2-1}%
\end{equation}
and
\begin{equation}
D_{X}Y=\frac{1}{2}\left[  X,Y\right] \label{2-2}%
\end{equation}
for all $X,Y$ and $Z\in \mathfrak{g}.$


Let $\alpha:I\subset \mathbb{R\rightarrow}G$ be an arc-lenghted regular curve
and $\left \{  X_{1},X_{2,}...,X_{n}\right \}  $ be an orthonormal basis of
$\mathfrak{g}.$ In this case, we write that any two vector fields $W$ and $Z$
along the curve $\alpha \ $as $W=\sum_{i=1}^{n}w_{i}X_{i}$ and $Z=\sum
_{i=1}^{n}z_{i}X_{i}$ where $w_{i}:I\rightarrow \mathbb{R}$ and $z_{i}%
:I\rightarrow \mathbb{R}$ are smooth functions. Also the Lie bracket of two
vector fields $W$ and $Z$ is given
\[
\left[  W,Z\right]  =\sum_{i=1}^{n}w_{i}z_{i}\left[  X_{i},X_{j}\right]
\]
and the covariant derivative of $W$ along the curve $\alpha$ with the notation
$D_{\alpha^{\shortmid}}W$ is given as follows%
\begin{equation}
D_{\alpha^{\shortmid}}W=\overset{\cdot}{W}+\frac{1}{2}\left[  T,W\right]
\label{2-3}%
\end{equation}
where $T=\alpha^{\prime}$ and $\overset{\cdot}{W}=\sum_{i=1}^{n}\overset
{\cdot}{w_{i}}X_{i}$ or $\overset{\cdot}{W}=\sum_{i=1}^{n}\frac{dw}{dt}X_{i}.$
Note that if $W$ is the left-invariant vector field to the curve $\alpha$ then
$\overset{\cdot}{W}=0$ (see for details \cite{crouch}).


Let $G$ be a three dimensional Lie group and $\left(  T,N,B,\varkappa
,\tau \right)  $ denote the Frenet apparatus of the curve $\alpha$. Then the
Serret-Frenet formulas of the curve $\alpha$ satisfies:%

\[
D_{T}T=\varkappa N\text{, \  \  \ }D_{T}N=-\varkappa T+\tau B\text{,
\  \  \ }D_{T}B=-\tau N
\]
where $D$ is Levi-Civita connection of Lie group $G$ and $\varkappa
=\overset{\cdot}{\left \Vert T\right \Vert }.$


\begin{definition}
\label{tan 2.1}Let $\alpha:I\subset \mathbb{R\rightarrow}G$ be a parametrized
curve. Then $\alpha$ is called a \emph{general helix} if it makes a constant
angle with a left-invariant vector field $X$. That is,%
\[
\left \langle T(s),X\right \rangle =\cos \theta \text{ for all }s\in I,
\]
for the left-invariant vector field $X\in g$ is unit length and $\theta$ is a
constant angle between $X$ and $T$, which is the tangent vector field of the
curve $\alpha$ (see \cite{ciftci}).
\end{definition}


\begin{definition}
\label{tan 2.2}Let $\alpha:I\subset \mathbb{R\rightarrow}G$ be a parametrized
curve with the Frenet apparatus $\left(  T,N,B,\varkappa,\tau \right)  $ then
\begin{equation}
\tau_{G}=\frac{1}{2}\left \langle \left[  T,N\right]  ,B\right \rangle
\label{2-4}%
\end{equation}
or
\[
\tau_{G}=\frac{1}{2\varkappa^{2}\tau}\overset{\cdot \cdot \text{
\  \  \  \  \  \  \  \ }\cdot}{\left \langle T,\left[  T,T\right]  \right \rangle
}+\frac{1}{4\varkappa^{2}\tau}\overset{\text{ \  \ }\cdot}{\left \Vert \left[
T,T\right]  \right \Vert ^{2}}%
\]
(see \cite{ciftci}).
\end{definition}


\begin{proposition}
\label{prop 2.1}Let $\alpha:I\subset \mathbb{R\rightarrow}G$ be an arc length
parametrized curve with the Frenet apparatus $\left \{  T,N,B\right \}  $. Then
the following equalities%
\begin{align*}
\left[  T,N\right]   &  =\left \langle \left[  T,N\right]  ,B\right \rangle
B=2\tau_{G}B\\
\left[  T,B\right]   &  =\left \langle \left[  T,B\right]  ,N\right \rangle
N=-2\tau_{G}N
\end{align*}
hold \cite{zeki}.
\end{proposition}

\begin{definition}
\label{tan 2.3}Let $\alpha:I\subset \mathbb{R\rightarrow}G$ be an arc length
parametrized curve. Then $\alpha$ is called a \emph{slant helix} if its
principal normal vector field makes a constant angle with a left-invariant
vector field $X$ which is unit length. That is,%
\[
\left \langle N(s),X\right \rangle =\cos \theta \text{ for all }s\in I,
\]
where $\theta \neq \frac{\pi}{2}$ is a constant angle between $X$ and $N$ which
is the principal normal vector field of the curve $\alpha$ (see \cite{zeki}).
\end{definition}

\begin{definition}
\label{tan 2.4}Let $\alpha:I\subset \mathbb{R\rightarrow}G$ be an arc length
parametrized curve with the Frenet apparatus $\left \{  T,N,B,\varkappa
,\tau \right \}  .$ Then the \emph{harmonic curvature function} of the curve
$\alpha$ is defined by%
\[
H=\dfrac{\tau-\tau_{G}}{\varkappa}%
\]
where $\tau_{G}=\frac{1}{2}\left \langle \left[  T,N\right]  ,B\right \rangle $
(see \cite{zeki})$.$

\begin{theorem}
\label{teo 2.1}Let $\alpha:I\subset \mathbb{R\rightarrow}G$ be a parametrized
curve with the Frenet apparatus $\left(  T,N,B,\varkappa,\tau \right)  $. If
the curve $\alpha$ is a general helix, if and only if,%
\[
\tau=c\varkappa+\tau_{G}%
\]
where c is a constant (see \cite{ciftci}) or using the definition of the
harmonic curvature function of the curve $\alpha$ (see \cite{zeki}) is
constant function.
\end{theorem}
\end{definition}

\begin{theorem}
\label{teo 2.2}Let $\alpha:I\subset \mathbb{R\rightarrow}G$ \ be a unit speed
curve with the Frenet apparatus $\left(  T,N,B,\varkappa,\tau \right)  $. Then
$\alpha$ is a slant helix if and only if%
\[
\sigma_{N}=\frac{\varkappa(1+H^{2})^{\frac{3}{2}}}{H^{\shortmid}}=\tan \theta
\]
is a constant where $H$ is a harmonic curvature function of the curve $\alpha$
and $\theta \neq \frac{\pi}{2}$ is a constant (see \cite{zeki}).
\end{theorem}

\section{Mannheim partner curves in a three dimensional Lie group}

In this section, we define Mannheim partner curves and their characterizations
are given in a three dimensional Lie group $G$ with a bi-invariant metric
$\left \langle \text{ },\right \rangle $. Also we give some characterizations of
Mannheim partner curves using the special cases of $G$.


\begin{definition}
\label{tan 3.1}A curve $\alpha$ in $3$-dimensional Lie group $G$ is a
\emph{Mannheim curve }if there exists a special curve\emph{\ }$\beta$ in
$3$-dimensional Lie group $G$ such that principal normal vector field of
$\alpha$ is linearly dependent binormal vector field of $\beta$ at
corresponding point under $\psi$ which is bijection from $\alpha$ to $\beta.$
In this case $\beta$ is called the \emph{Mannheim partner} $\emph{curve}$ of
$\alpha$ and $\left(  \alpha,\beta \right)  $ is called \emph{Mannheim}
$\emph{curve}$\emph{\ couple.}
\end{definition}

The curve $\alpha:I\subset \mathbb{R\rightarrow}G$ in $3$-dimensional Lie group
$G$ is parametrized by the arc-length parameter $s$ and from the Definition
\ref{tan 3.1} Bertrand mate curve of $\alpha$ is given $\beta:\overline
{I}\subset \mathbb{R\rightarrow}G$ in $3$-dimensional Lie group $G$ with the
help of Figure 1 such that%
\[
\text{\texttt{Figure1}}\mathtt{:}\text{{}Mannheim curve\ couple }\left(
\alpha,\beta \right)
\]%
\[
\beta \left(  s\right)  =\alpha \left(  s\right)  +\lambda \left(  s\right)
N\left(  s\right)  ,\text{ }s\in I
\]
where $\lambda$ is a smooth function on $I$ and $N$ is the principal normal
vector field of $\alpha$.

\begin{theorem}
\label{teo 3.1}Let $\alpha:I\subset \mathbb{R\rightarrow}G$ and $\beta
:\overline{I}\subset \mathbb{R\rightarrow}G$ be a Mannheim curve couple with
arc-length parameter $s$ and $\overline{s},$ respectively. Then corresponding
points are a fixed distance apart for all $s\in I$, that is,
\[
d\left(  \alpha \left(  s\right)  ,\beta \left(  s\right)  \right)
=\text{constant, \  \ for all }s\in I
\]

\end{theorem}

\begin{proof}
From Definition \ref{tan 3.1}, we can simply write
\begin{equation}
\beta \left(  s\right)  =\alpha \left(  s\right)  +\lambda \left(  s\right)
N\left(  s\right) \label{3-2}%
\end{equation}
Differentiating the Eq. \eqref{3-2} with respect to $s$ and using the Eq.
\eqref{2-3}, we get%
\begin{align*}
\frac{d\beta \left(  \overline{s}\right)  }{d\overline{s}}\frac{d\overline{s}%
}{ds}  & =\frac{d\alpha \left(  s\right)  }{ds}+\lambda^{\prime}\left(
s\right)  N\left(  s\right)  +\lambda(s)\overset{\cdot}{N}(s)\\
& =T\left(  s\right)  +\lambda^{\prime}\left(  s\right)  N\left(  s\right)
+\lambda(s)\left[  D_{T}N-\frac{1}{2}\left[  T,N\right]  \right]
\end{align*}
and with the help of Proposition \ref{prop 3.1} and Frenet equations, we
obtain
\[
\frac{d\beta \left(  \overline{s}\right)  }{d\overline{s}}\frac{d\overline{s}%
}{ds}=\left(  1-\lambda \left(  s\right)  \varkappa \left(  s\right)  \right)
T(s)+\lambda^{\prime}\left(  s\right)  N\left(  s\right)  +\lambda \left(
s\right)  \left(  \left(  \tau-\tau_{G}\right)  \left(  s\right)  \right)
B\left(  s\right)
\]
or%
\begin{equation}
T_{\beta}\left(  \overline{s}\right)  =\frac{ds}{d\overline{s}}\left[  \left(
1-\lambda(s)\varkappa \left(  s\right)  \right)  T\left(  s\right)
+\lambda^{\prime}\left(  s\right)  N\left(  s\right)  +\lambda(s)\left(
\tau-\tau_{G}\right)  (s)B\left(  s\right)  \right]  .\label{3-3}%
\end{equation}
And then, we know that $\left \{  N\left(  s\right)  ,B_{\beta}\left(
\overline{s}\right)  \right \}  $ is a linearly dependent set, so we have
\[
\left \langle T_{\beta}\left(  \overline{s}\right)  ,B_{\beta}\left(
\overline{s}\right)  \right \rangle =\frac{ds}{d\overline{s}}\left[
\begin{array}
[c]{c}%
\left(  1-\lambda \left(  s\right)  \varkappa \left(  s\right)  \right)
\left \langle T(s),B_{\beta}\left(  \overline{s}\right)  \right \rangle
+\lambda^{\prime}\left(  s\right)  \left \langle N(s),B_{\beta}\left(
\overline{s}\right)  \right \rangle \\
+\lambda(s)\left(  \tau-\tau_{G}\right)  (s)\left \langle B(s),B_{\beta
}(\left(  \overline{s}\right)  )\right \rangle
\end{array}
\right]
\]
or%
\[
\lambda^{\prime}\left(  s\right)  =0
\]
that is, $\lambda \left(  s\right)  $ is constant function on $I.$ This
completes the proof.
\end{proof}

\begin{theorem}
\label{teo 3.2}Let $\alpha:I\subset \mathbb{R\rightarrow}G$ \ be a parametrized
curve with arc length parameter $s$ and the Frenet apparatus $\left(
T,N,B,\varkappa,\tau \right)  $. Then, $\alpha$ is Mannheim curve if and only
if
\begin{equation}
\lambda \varkappa \left(  1+H^{2}\right)  =1,\text{ for all }s\in I\label{3-4}%
\end{equation}
where $\lambda$ is constant and $H$ is the harmonic curvature function of the
curve $\alpha.$
\end{theorem}

\begin{proof}
Let $\alpha:I\subset \mathbb{R\rightarrow}G$ be a parametrized Mannheim curve
with arc length parameter $s$ then we can write
\[
\beta \left(  s\right)  =\alpha \left(  s\right)  +\lambda N\left(  s\right)
\]
Differentiating the above equality with respect to $s$ and by using the Frenet
equations, we get%
\[
\frac{d\beta \left(  s\right)  }{ds}=\left(  1-\lambda \varkappa \left(
s\right)  \right)  T\left(  s\right)  +\lambda \left(  \tau-\tau_{G}\right)
(s)B\left(  s\right)  .
\]
On the other hand, we have%
\[
T_{\beta}=\frac{d\beta}{ds}\frac{ds}{d\overline{s}}=\left[  \left(
1-\lambda \varkappa \left(  s\right)  \right)  T\left(  s\right)  +\lambda
\left(  \tau-\tau_{G}\right)  (s)B\left(  s\right)  \right]  \frac
{ds}{d\overline{s}}.
\]
By taking the derivative of this equation with respect to $\overline{s}$ and
using the Frenet equations we obtain%
\begin{align*}
\frac{dT_{\beta}}{d\overline{s}}  & =\left[  -\lambda \frac{d\varkappa}%
{ds}T(s)+\left(  \varkappa-\lambda \varkappa^{2}-\lambda \left(  \tau-\tau
_{G}\right)  ^{2}\right)  N(s)+\lambda \left(  \tau-\tau_{G}\right)
^{\shortmid}B(s)\right]  \left(  \frac{ds}{d\overline{s}}\right)  ^{2}\\
& +\left[  \left(  1-\lambda \varkappa \left(  s\right)  \right)  T\left(
s\right)  +\lambda \left(  \tau-\tau_{G}\right)  (s)B\left(  s\right)  \right]
\frac{d^{2}s}{d\overline{s}^{2}}%
\end{align*}
From this equation we get
\[
\left(  \varkappa-\lambda \varkappa^{2}-\lambda \left(  \tau-\tau_{G}\right)
^{2}\right)  =0,
\]%
\[
\lambda \varkappa \left(  1+H^{2}\right)  =1.
\]
Conversely, if $\lambda \varkappa \left(  1+H^{2}\right)  =1$ then we can easily
see that $\alpha$ is a Mannheim curve.

This completes proof.
\end{proof}

\begin{corollary}
If $G$ is Abellian Lie group then $\tau_{G}=0.$ And so, if $\alpha
:I\subset \mathbb{R\rightarrow}G$ \ is a parametrized Mannheim curve with arc
length parameter $s$ and the Frenet apparatus $\left(  T,N,B,\varkappa
,\tau \right)  $ in Abellian Lie group $G$. Then, $\alpha$ satisfy the
following equality%
\[
\lambda \left(  \varkappa^{2}+\tau^{2}\right)  =\varkappa
\]
%

\end{corollary}

\begin{proof}
If $G$ is Abellian Lie group then using the $\tau_{G}=0$ and Theorem
\ref{teo 3.2} we have the result.
\end{proof}


So, the above Corollary shows that the study is a generalization of Mannheim
curves defined by Liu and Wang \cite{liu} in Euclidean 3-space.


\begin{theorem}
\label{teo 3.2.5}Let $\alpha:I\subset \mathbb{R\rightarrow}G$ \ be a
parametrized curve with arc length parameter $s$. Then $\beta$ is the Mannheim
partner curve of $\alpha$ if and only if the curvature $\varkappa_{\beta}$ and
the torsion $\tau_{\beta}$ of $\beta$ satisfy the following equation
\[
\frac{d\varkappa_{\beta}H_{\beta}}{d\overline{s}}=\frac{\varkappa_{\beta}}%
{\mu}(1+\mu^{2}\varkappa_{\beta}^{2}H_{\beta}^{2})
\]
where $\mu$ is constant and $H_{\beta}$ is the harmonic curvature function of
the curve $\beta.$
\end{theorem}


\begin{proof}
Let $\alpha:I\subset \mathbb{R\rightarrow}G$ be a parametrized Mannheim curve
with arc length parameter $s$ then we can write%
\[
\alpha \left(  \overline{s}\right)  =\beta \left(  \overline{s}\right)
+\mu \left(  \overline{s}\right)  B_{\beta}\left(  \overline{s}\right)
\]
for some function $\mu \left(  \overline{s}\right)  $. By taking the derivative
of this equation with respect to $\overline{s}$ and using the Frenet equations
we obtain%
\[
T\frac{ds}{d\overline{s}}=T_{\beta}+\mu^{\shortmid}\left(  \overline
{s}\right)  B_{\beta}\left(  \overline{s}\right)  -\mu \left(  \overline
{s}\right)  \left(  \tau_{\beta}-\tau_{G_{\beta}}\right)  \left(  \overline
{s}\right)  N_{\beta}\left(  \overline{s}\right)
\]
or%
\[
T\frac{ds}{d\overline{s}}=T_{\beta}+\frac{d\mu \left(  \overline{s}\right)
}{d\overline{s}}B_{\beta}\left(  \overline{s}\right)  -\mu \left(  \overline
{s}\right)  \varkappa_{\beta}H_{\beta}N_{\beta}\left(  \overline{s}\right)
\]
where $H_{\beta}$ is the harmonic curvature function of the curve $\beta.$ And
then, we know that $\left \{  N\left(  s\right)  ,B_{\beta}\left(  \overline
{s}\right)  \right \}  $ is a linearly dependent set, so we have
\[
\frac{d\mu \left(  \overline{s}\right)  }{d\overline{s}}=0.
\]
This means that $\mu \left(  \overline{s}\right)  $ is a constant function.
Thus we have%
\begin{equation}
T\frac{ds}{d\overline{s}}=T_{\beta}-\mu \left(  \overline{s}\right)
\varkappa_{\beta}H_{\beta}N_{\beta}\left(  \overline{s}\right)  .\label{3-5}%
\end{equation}
On the other hand, we have%
\begin{equation}
T=T_{\beta}\cos \theta+N_{\beta}\sin \theta \label{3-6}%
\end{equation}
where $\theta$ is the angle between $T$ and $T_{\beta}$ at the corresponding
points of $\alpha$ and $\beta.$ By taking the derivative of this equation with
respect to $\overline{s}$ and using the Frenet equations we obtain%
\[
\varkappa N\frac{ds}{d\overline{s}}=-\left(  \varkappa_{\beta}+\frac{d\theta
}{d\overline{s}}\right)  \sin \theta T_{\beta}+\left(  \varkappa_{\beta}%
+\frac{d\theta}{d\overline{s}}\right)  \cos \theta N_{\beta}+\varkappa_{\beta
}H_{\beta}\sin \theta B_{\beta}.
\]
From this equation and the fact that the $\left \{  N\left(  s\right)
,B_{\beta}\left(  \overline{s}\right)  \right \}  $ is a linearly dependent
set, we get%
\[
\left \{
\begin{array}
[c]{c}%
\left(  \varkappa_{\beta}+\frac{d\theta}{d\overline{s}}\right)  \sin \theta=0\\
\left(  \varkappa_{\beta}+\frac{d\theta}{d\overline{s}}\right)  \cos \theta=0.
\end{array}
\right.
\]
For this reason we have%
\begin{equation}
\frac{d\theta}{d\overline{s}}=-\varkappa_{\beta}.\label{3-7}%
\end{equation}
From the Eq. \eqref{3-5} and Eq. \eqref{3-6} and notice that $T_{\beta}$ is
orthogonal to $B_{\beta}$, we find that%
\[
\frac{ds}{d\overline{s}}=\frac{1}{\cos \theta}=-\frac{\mu \varkappa_{\beta
}H_{\beta}}{\sin \theta}.
\]
Then we have%
\[
\mu \varkappa_{\beta}H_{\beta}=-\tan \theta.
\]
By taking the derivative of this equation and applying Eq. \eqref{3-7}, we get%
\[
\mu \frac{d\varkappa_{\beta}H_{\beta}}{d\overline{s}}=\varkappa_{\beta}\left(
1+\mu^{2}\varkappa_{\beta}^{2}H_{\beta}^{2}\right)
\]
that is%
\[
\frac{d\varkappa_{\beta}H_{\beta}}{d\overline{s}}=\frac{\varkappa_{\beta}}%
{\mu}\left(  1+\mu^{2}\varkappa_{\beta}^{2}H_{\beta}^{2}\right)  .
\]
Conversely, if the curvature $\varkappa_{\beta}$ and torsion $\tau_{\beta}$ of
the curve $\beta$ in three dimensional Lie group $G$ satisfy%
\[
\frac{d\varkappa_{\beta}H_{\beta}}{d\overline{s}}=\frac{\varkappa_{\beta}}%
{\mu}\left(  1+\mu^{2}\varkappa_{\beta}^{2}H_{\beta}^{2}\right)
\]
for constant $\mu \left(  \overline{s}\right)  ,$ then we define a curve a
curve by%
\begin{equation}
\alpha \left(  \overline{s}\right)  =\beta \left(  \overline{s}\right)  +\mu
B_{\beta}\left(  \overline{s}\right) \label{3-8}%
\end{equation}
and we will show that $\alpha$ is a Mannheim curve and $\beta$ is the partner
curve of $\alpha$ in three dimensional Lie group $G$. By taking the derivative
of Eq. \eqref{3-8} with respect to $\overline{s}$ twice, we get%
\begin{equation}
T\frac{ds}{d\overline{s}}=T_{\beta}-\mu \varkappa_{\beta}H_{\beta}N_{\beta
},\label{3-9}%
\end{equation}%
\begin{equation}
\varkappa N\left(  \frac{ds}{d\overline{s}}\right)  ^{2}+T\frac{d^{2}%
s}{d\overline{s}^{2}}=\mu \varkappa_{\beta}^{2}H_{\beta}T_{\beta}+\left(
\varkappa_{\beta}-\mu \frac{d\varkappa_{\beta}H_{\beta}}{d\overline{s}}\right)
N_{\beta}-\mu \varkappa_{\beta}^{2}H_{\beta}^{2}B_{\beta},\label{3-10}%
\end{equation}
respectively. Taking the cross product of Eq. \eqref{3-9} with Eq.
\eqref{3-10} and noticing that%
\[
\varkappa_{\beta}-\mu \frac{d\varkappa_{\beta}H_{\beta}}{d\overline{s}}+\mu
^{2}\varkappa_{\beta}^{3}H_{\beta}^{2}%
\]
we have%
\begin{equation}
\varkappa B\left(  \frac{ds}{d\overline{s}}\right)  ^{3}=\mu^{2}%
\varkappa_{\beta}^{3}H_{\beta}^{3}T_{\beta}+\mu \varkappa_{\beta}^{2}H_{\beta
}^{2}N_{\beta}.\label{3-11}%
\end{equation}
By taking the cross product of Eq. \eqref{3-9} with Eq. \eqref{3-11}, we get%
\[
\varkappa N\left(  \frac{ds}{d\overline{s}}\right)  ^{4}=-\mu \varkappa_{\beta
}^{2}H_{\beta}^{2}\left(  1+\mu^{2}\varkappa_{\beta}^{2}H_{\beta}^{2}\right)
B_{\beta}.
\]
This means that the principal normal vector field of the curve $\alpha$ and
binormal vector field of the curve $\beta$ are linearly dependent set. And so
$\alpha$ is a Mannheim curve and $\beta$ is Mannheim partner curve of the
curve $\alpha$ in three dimensional Lie group $G.$
\end{proof}


\begin{proposition}
\label{prop 3.1}Let $\alpha:I\subset \mathbb{R\rightarrow}G$\ be an
arc-lenghted Mannheim curve with the Frenet vector fields $\left \{
T,N,B\right \}  $ and $\beta:\overline{I}\subset \mathbb{R\rightarrow}G$ be a
Mannheim mate of $\alpha$ with the Frenet vector fields $\left \{  T_{\beta
},N_{\beta},B_{\beta}\right \}  .$ Then $\tau_{G}=\tau_{G_{\beta}}$ for the
curves $\alpha$ and $\beta$ where $\tau_{G}=\frac{1}{2}\left \langle \left[
T,N\right]  ,B\right \rangle $ and $\tau_{G_{\beta}}=\frac{1}{2}\left \langle
\left[  T_{\beta},N_{\beta}\right]  ,B_{\beta}\right \rangle .$
\end{proposition}


\begin{proof}
Let $\left(  \alpha,\beta \right)  $ be a Mannheim curve couple. From the Eq.
\eqref{3-3} we have%
\[
T_{\beta}\left(  \overline{s}\right)  =\left[  \left(  1-\lambda
\varkappa \left(  s\right)  \right)  T\left(  s\right)  +\lambda \left(
\tau-\tau_{G}\right)  (s)B\left(  s\right)  \right]  \frac{ds}{d\overline{s}}.
\]
We take the norm this equation and by using Eq. \eqref{3-4}, we obtain%
\[
\frac{d\overline{s}}{ds}=\lambda \varkappa H\sqrt{1+H^{2}}.
\]
If we think together the last two equations, we get%
\begin{equation}
T_{\beta}\left(  \overline{s}\right)  =\frac{H}{\sqrt{1+H^{2}}}T(s)+\frac
{1}{\sqrt{1+H^{2}}}B(s).\label{3-12}%
\end{equation}
Since $\left(  \alpha,\beta \right)  $ is a Mannheim curve couple we know
$B_{\beta}(\overline{s})=N(s).$ Then,%
\[
N_{\beta}(\overline{s})=B_{\beta}(\overline{s})\times T_{\beta}(\overline{s})
\]%
\[
N_{\beta}(\overline{s})=\frac{1}{\sqrt{1+H^{2}}}T(s)-\frac{H}{\sqrt{1+H^{2}}%
}B(s).
\]
We know from Definition \ref{tan 2.2} $\left \langle \left[  T_{\beta}%
,N_{\beta}\right]  ,B_{\beta}\right \rangle =2\tau_{G_{\beta}}$ for the curve
$\beta.$ Then with the help of above equations for $T_{\beta}\left(
\overline{s}\right)  ,N_{\beta}(\overline{s})$ and $B_{\beta}(\overline{s})$,
we obtain%
\[
\left \langle \left[  \frac{H}{\sqrt{1+H^{2}}}T+\frac{1}{\sqrt{1+H^{2}}}%
B,\frac{1}{\sqrt{1+H^{2}}}T(s)-\frac{H}{\sqrt{1+H^{2}}}B(s)\right]
,N\right \rangle =2\tau_{G_{\beta}},
\]%
\[
\frac{H^{2}}{\sqrt{1+H^{2}}}\left \langle \left[  T,N\right]  ,B\right \rangle
+\frac{1}{\sqrt{1+H^{2}}}\left \langle \left[  T,N\right]  ,B\right \rangle
=2\tau_{G_{\beta}}.
\]
By using the Proposition \ref{prop 2.1} in last equation, we get%
\[
\tau_{G}=\tau_{G_{\beta}.}%
\]
This completes proof.
\end{proof}

\begin{theorem}
\label{teo 3.3}Let $\alpha:I\subset \mathbb{R\rightarrow}G$ \ be a parametrized
Mannheim curve with curvature functions $\varkappa$, $\tau$ and $\beta
:\overline{I}\subset \mathbb{R\rightarrow}G$ be a Mannheim mate of $\alpha$
with curvatures functions $\varkappa_{\beta}$, $\tau_{\beta}.$ Then the
relations between these curvature functions are%
\begin{align*}
\varkappa_{\beta}\left(  \overline{s}\right)   & =\frac{H^{\shortmid}\left(
s\right)  }{\lambda \varkappa(s)H(s)\left(  1+H^{2}(s)\right)  ^{3/2}},\\
\tau_{\beta}\left(  \overline{s}\right)   & =\frac{1}{\lambda H\left(
s\right)  }+\tau_{G_{\beta}.}%
\end{align*}

\end{theorem}

\begin{proof}
If we differentiating the Eq. \eqref{3-12}\ and using the Frenet formulas, we
have%
\[
\varkappa_{\beta}N_{\beta}\lambda \varkappa H\sqrt{1+H^{2}}=\frac{H^{\shortmid
}}{\left(  1+H^{2}(s)\right)  ^{3/2}}\left(  T-HB\right)  .
\]
If we take the norm of last equation, we get%
\[
\varkappa_{\beta}=\frac{H^{\shortmid}}{\lambda \varkappa H\left(
1+H^{2}(s)\right)  ^{3/2}}.
\]
Since $\left(  \alpha,\beta \right)  $ is a Mannheim curve couple, we know
$B_{\beta}=N.$ If we differentiating this equation and using the Frenet
formulas, we have%
\[
-\left(  \tau_{\beta}-\tau_{G_{\beta}}\right)  N_{\beta}\lambda H\sqrt
{1+H^{2}}=-T+HB.
\]
If we take the norm of last equation, we get%
\[
\tau_{\beta}=\frac{1}{\lambda H}+\tau_{G_{\beta}}.
\]
This completes proof.
\end{proof}


\begin{theorem}
\label{teo 3.4}Let $\alpha:I\subset \mathbb{R\rightarrow}G$\ be an arc-lenghted
Mannheim curve and $\beta:\overline{I}\subset \mathbb{R\rightarrow}G$ be a
Mannheim mate of $\alpha.$ The Mannheim curve $\alpha$ is a slant helix if and
only if its Mannheim mate $\beta$ is a general helix.
\end{theorem}


\begin{proof}
If Mannheim curve $\alpha$ is a slant helix, than we have Theorem
\ref{teo 2.2} $\sigma_{N}$ is a constant function. From Theorem \ref{teo 3.3}
for the curve $\beta,$ we have%
\begin{align*}
\frac{\tau_{\beta}-\tau_{G_{\beta}}}{\varkappa_{\beta}}  & =\frac{\frac
{1}{\lambda H}}{\frac{H^{\shortmid}}{\lambda \varkappa H\left(  1+H^{2}\right)
^{3/2}}}\\
\frac{\tau_{\beta}-\tau_{G_{\beta}}}{\varkappa_{\beta}}  & =\frac
{\varkappa \left(  1+H^{2}\right)  ^{3/2}}{H^{\shortmid}}\\
\frac{\tau_{\beta}-\tau_{G_{\beta}}}{\varkappa_{\beta}}  & =\sigma
_{N}=\text{constant.}%
\end{align*}
So $\beta$, which is Mannheim mate of $\alpha,$is a general helix.

Conversely, we assume that $\beta$, which is Mannheim mate of $\alpha,$ be a
general helix. So we have
\[
\frac{\tau_{\beta}-\tau_{G_{\beta}}}{\varkappa_{\beta}}=\text{constant.}%
\]
From last equation $\sigma_{N}$ is a constant function. This completes the proof.
\end{proof}

\begin{theorem}
\label{teo 3.5}Let $\alpha:I\subset \mathbb{R\rightarrow}G$\ be an arc-lenghted
Mannheim curve and $\beta:\overline{I}\subset \mathbb{R\rightarrow}G$ be a
Mannheim mate of $\alpha$. If $\alpha$ is a slant helix then the harmonic
curvature function of the curve $\alpha$: $H(s)$ is
\begin{equation}
H(s)=\frac{1}{2}\left(  ae^{bs}-\frac{1}{a}e^{-bs}\right) \label{3-13}%
\end{equation}
for some nonzero constant $a$ and $b$ and $s$ is the arc length parameter of
the curve $\alpha$. If we consider $a=b=1,$ we have the harmonic curvature
funtion of the curve $\alpha$ is hyperbolic sine function in arc length $s$,
that is., $H(s)=\sinh s$.
\end{theorem}

\begin{proof}
Let $\alpha:I\subset \mathbb{R\rightarrow}G$\ be an arc-lenghted Mannheim curve
with Frenet apparatus $\left \{  T,N,B,\varkappa,\tau \right \}  $ in three
dimensional Lie group. Assume that $\alpha$ be slant helix, we have%
\begin{equation}
\left \langle N,X\right \rangle =\cos \theta,\text{ }\theta \neq \frac{\pi}%
{2}\label{3-14}%
\end{equation}
for left invariant vector field $X.$ Differentiating the Eq. \eqref{3-14}
twice, we have%
\begin{equation}
-\varkappa \left \langle T,X\right \rangle +\left(  \tau-\tau_{G}\right)
\left \langle B,X\right \rangle =0\label{3-15}%
\end{equation}
and%
\[
-\varkappa^{\shortmid}\left \langle T,X\right \rangle +\left(  \tau-\tau
_{G}\right)  ^{\shortmid}\left \langle B,X\right \rangle =\left \{  \varkappa
^{2}+\left(  \tau-\tau_{G}\right)  ^{2}\right \}  \left \langle N,X\right \rangle
.
\]
Since $\alpha$ is a Mannheim curve using the Theorem \ref{teo 3.2}, we rewrite
the last equation%
\begin{equation}
-\varkappa^{\shortmid}\left \langle T,X\right \rangle +\left(  \tau-\tau
_{G}\right)  ^{\shortmid}\left \langle B,X\right \rangle =\frac{\varkappa
}{\lambda}\cos \theta \label{3-16}%
\end{equation}
where $\lambda$ is a non-zero constant. By a direct calculation using the Eq.
\eqref{3-15} and the Eq. \eqref{3-16}, we obtain%
\begin{equation}
\left \langle T,X\right \rangle =\frac{H}{\lambda H^{\shortmid}}\cos
\theta \label{3-17}%
\end{equation}
and%
\begin{equation}
\left \langle B,X\right \rangle =\frac{1}{\lambda H^{\shortmid}}\cos
\theta.\label{3-18}%
\end{equation}
Differentiating the Eq. \eqref{3-17} and the Eq. \eqref{3-18}, we have%
\[
\varkappa=\frac{1}{\lambda}\left(  1-\frac{HH^{\shortmid \shortmid}}{\left(
H^{\shortmid}\right)  ^{2}}\right)  ,
\]%
\[
\tau-\tau_{G}=\frac{H^{\shortmid \shortmid}}{\lambda \left(  H^{\shortmid
}\right)  ^{2}},
\]
respectively. These equations give that%
\[
H=\frac{\tau-\tau_{G}}{\varkappa}=\frac{H^{\shortmid \shortmid}}{\left(
H^{\shortmid}\right)  ^{2}-HH^{\shortmid \shortmid}}.
\]
Then we have the following differential equation%
\[
\left(  1+H^{2}\right)  H^{\shortmid \shortmid}-\left(  H^{\shortmid}\right)
^{2}=0.
\]
Solving the last equation, we obtain the Eq. \eqref{3-13}. This completes the proof.
\end{proof}

\begin{theorem}
\label{teo 3.6}Let $\left(  \alpha,\beta \right)  $ be a Mannheim curve couple
in three dimensional Lie group with bi-invariant metric. Then $\alpha$ is
general helix if and only if $\beta$ is a geodesic.
\end{theorem}

\begin{proof}
If Mannheim curve $\alpha$ is a general helix, then its harmonic curvature $H$
is constant function. And so from Theorem \ref{teo 3.3},
\[
\varkappa_{\beta}=0.
\]
So, $\beta$ is a geodesic.

Conversely we assume that $\beta$ be a geodesic curve. From Theorem
\ref{teo 3.3} we can easily see that%
\[
H^{\shortmid}\left(  s\right)  =0
\]
and so%
\[
H(s)=\text{constant.}%
\]
This comletes the proof.
\end{proof}

\end{document}